\documentclass[a4paper,10pt]{article}%
\usepackage{amsmath}
\usepackage{amsfonts}
\usepackage{amssymb}
\usepackage{amsthm}
\usepackage[a4paper,scale=0.85]{geometry}
\newtheorem{theorem}{Theorem}[section]
\newtheorem{corollary}{Corollary}[section]
\newtheorem{lemma}{Lemma}[section]
\newtheorem{remark}{Remark}[section]
\newtheorem{example}{Example}[section]
\begin{document}

\title{The fixed point problem for systems of coordinate-wise monotone operators and applications}
\author{Mircea-Dan Rus\thanks{Department of Mathematics, Faculty of Automation and Computer Science, Technical University of Cluj-Napoca, 400027 Cluj-Napoca, Romania; e-mail: rus.mircea@math.utcluj.ro}}
\maketitle

\begin{abstract}
We study the fixed point problem for a system of multivariate operators that are coordinate-wise monotone (i.e., nondecreasing or nonincreasing in each of the variables, independently), in the setting of quasi-ordered sets. We show that this problem is equivalent to the fixed point problem for a mixed monotone operator that can be explicitly constructed. As a consequence, we obtain a criterion for the existence and uniqueness of solution to our problem, in the setting of partially ordered metric spaces. To validate our results, we provide an application to a first-order differential system with periodic boundary value conditions. A direct consequence that follows from our paper is that all the separate recent developments on the subject of tripled, quadrupled or multidimensional fixed points are but particular aspects of a single, unified and much simpler approach.\bigskip

\bigskip

\noindent \textbf{Mathematics Subject Classification (MSC 2000):} 47H10; 34B15.
\end{abstract}

\section{Introduction}

Let $N\geq2$ be a positive integer, $\mathcal{X}=\left\{  (X_{i},\leq
_{i}):i\in\{1,2,\ldots,N\}\right\}  $ a family of quasi-ordered sets (i.e.,
$\leq_{i}$ is a reflexive and transitive relation on set $X_{i}$) and
$\mathcal{T}=\left\{  T_{i}:X_{1}\times X_{2}\times\cdots\times X_{N}%
\rightarrow X_{i}:i\in\{1,2,\ldots,N\}\right\}  $ a family of operators. The
aim of this paper is to study the system%
\begin{equation}
x_{i}=T_{i}(x_{1},x_{2},\ldots,x_{N}),\quad i\in\{1,2,\ldots,N\}\quad\left(
(x_{1},x_{2},\ldots,x_{N})\in X_{1}\times X_{2}\times\cdots\times
X_{N}\right)  ,~ \label{eq:rus1207-01}%
\end{equation}
under the assumption that, for every $i\in\{1,2,\ldots,N\}$, $T_{i}$ is
\emph{coordinate-wise monotone} (i.e., for every $j$ $\in\{1,2,\ldots,N\}$,
$T_{i}$ is a nondecreasing or nonincreasing mapping of $x_{j}$, when all the
other variables are fixed). In this context, $\left(  \mathcal{X}%
,\mathcal{T}\right)  $ will be called \emph{a partially monotone system
}(\emph{of order }$N$), and (\ref{eq:rus1207-01}) will be referred to as
\emph{the fixed point problem for the partially monotone system }$\left(
\mathcal{X},\mathcal{T}\right)  $. It is worth mentioning that such systems
can model the dynamics of many chemical, physical, biological or financial processes.

This problem has been previously studied by the author in \cite{Rus2008}, in
the framework of ordered Banach spaces. Also, recently, Turinici
\cite{Turinici2011} investigated (\ref{eq:rus1207-01}) in the setting of
quasi-ordered metric spaces, under the assumption that all the operators in
$\mathcal{T}$ are nondecreasing (in each of the arguments), by studying the
fixed points of the nondecreasing operator
\begin{equation}
T=(T_{1},T_{2},\ldots,T_{N}):X\rightarrow X\qquad(X:=X_{1}\times X_{2}%
\times\cdots\times X_{N}), \label{eq:rus1207-01a}%
\end{equation}
subject to the usual product quasi-order on $X$
\begin{equation}
(x_{1},x_{2},\ldots,x_{N})\leq(y_{1},y_{2},\ldots,y_{N})\Longleftrightarrow
x_{i}\leq_{i}y_{i}\text{ for all }i\in\{1,2,\ldots,N\}\text{.}
\label{eq:rus1207-02}%
\end{equation}
Though significant, the case analyzed by Turinici \cite{Turinici2011} is
rather limited and does not provide an answer to (\ref{eq:rus1207-01}) in
general (see Remark \ref{th:rus1207-07}).

In this context, we prove that for any partially monotone system $\left(
\mathcal{X},\mathcal{T}\right)  $, the associated operator $T$ defined in
(\ref{eq:rus1207-01a}) is \emph{heterotone} (cf. \cite{Opoitsev1978a}), i.e.,
it can be expressed as%
\begin{equation}
T(x)=A(x,x)\quad\text{for all }x\in X\text{,} \label{eq:rus1207-03}%
\end{equation}
where $A:X^{2}\rightarrow X$ is some \emph{mixed-monotone} operator (i.e.,
nondecreasing in the first argument and nonincreasing in the second argument;
cf. \cite{Guo1987}); furthermore, we find $A$ explicitly, in terms of the
operators in $\mathcal{T}$ (Theorem \ref{th:rus1207-04}). In this way, we can
establish an effective equivalence between (\ref{eq:rus1207-01}) and the fixed
point problem for $A$. We further use this fact to obtain a criterion for the
existence and uniqueness of solutions to (\ref{eq:rus1207-01}) in the
framework of ordered metric spaces (see Theorem \ref{th:rus1207-05}) via a new
fixed point theorem for mixed monotone operators; as a particular case, we
study an abstract tripled fixed point problem. We also provide an application
to a periodic boundary value system to validate our results.

\section{Preliminaries}

Let $(X,\leq)$ be a quasi-ordered set. If $Y$ is a nonempty subset of $X$ and
$x\in X$, then $x$ is called \emph{a lower bound for }$Y$ (and $Y$ is said to
be \emph{bounded from below by }$x$) if $x\leq y$ for all $y\in Y$; also, $x$
is called \emph{an upper bound for }$Y$ (and $Y$ is said to be \emph{bounded
from above by }$x$) if $y\leq x$ for all $y\in Y$. We say that $(X,\leq)$ is
\emph{quasi-directed} if every two-element subset of $X$ has a lower bound or
an upper bound; also, $(X,\leq)$ is said to be \emph{bi-directed} if every
two-element subset of $X$ has both a lower bound and an upper bound.

Now, let $X,Y,Z$ be nonempty sets, and $A:X^{2}\rightarrow Y$, $B:$
$Y^{2}\rightarrow Z$ two bivariate operators. The symmetric composition (or,
the $s$\emph{-composition} for short) of $A$ and $B$ is defined by%
\[
B\ast A:X^{2}\rightarrow Z,\quad(B\ast A)(x,y)=B(A(x,y),A(y,x))\quad(x,y\in
X).
\]
(cf. \cite{Rus2011}). Also, for each nonempty set $E$, denote by $P_{E}$ the
projection mapping%
\[
P_{E}:E^{2}\rightarrow E,\quad P(x,y)=x\quad(x,y\in E).
\]
The $s$-composition is an associative law, while $P_{X}$ and $P_{Y}$ are the
right- and the left identity elements, respectively (i.e., $A\ast P_{X}%
=P_{Y}\ast A=A$). In addition, if $X,Y,Z$ are quasi-ordered sets and $A,B$ are
mixed monotone, then $B\ast A$ is also mixed monotone. Consequently, if $A$ is
a bivariate self-map of $X$ (i.e., $A:X^{2}\rightarrow X$), then one can
define the functional powers (i.e., the iterates) of $A$ with respect to the
$s$-composition by%
\[
A^{n+1}=A\ast A^{n}=A^{n}\ast A\quad(n=0,1,2,...),\quad A^{0}=P_{X}\text{.}%
\]
Moreover, if $X$ is a quasi-ordered set and $A$ is mixed monotone, then
$A^{n}$ is mixed monotone for every $n$. More details can be found in
\cite{Rus2011}.

Recall also (cf. \cite{Guo1987}) that a pair $(x,y)\in X^{2}$ is called
\emph{a coupled fixed point} of an operator $A:X^{2}\rightarrow X$ if
$A(x,y)=x$ and $A(y,x)=y$. Also, $x$ is called a \emph{fixed point} of $A$ if
$A(x,x)=x$.

Let us also consider the following class of functions%
\[
\Phi=\left\{  \varphi:[0,\infty)\rightarrow\lbrack0,\infty),~\varphi\text{ is
nondecreasing and }\varphi^{n}(t)\rightarrow 0\text{ as }n\rightarrow
\infty\text{ for all }t\geq0\right\}  \text{,}%
\]
where $\varphi^{n}$ denotes the $n$-th iterate of function $\varphi$. The
elements of $\Phi$ are sometimes referred to as \emph{comparison functions}.
For example, $\varphi(t)=\alpha t$ (where $\alpha\in\lbrack0,1)$),
$\varphi(t)=\ln(1+t)$ and $\varphi(t)=\frac{t}{t+1}$ ($t\geq0$) are such
functions (cf. \cite{O'Regan2008}).

\section{Auxiliary results}

In order to prove our results, we indirectly need the following fixed point
theorem for nondecreasing operators, due to O'Regan and Petru\c{s}el
\cite{O'Regan2008}.

\begin{theorem}
[O'Regan and Petru\c{s}el, \cite{O'Regan2008}]\label{th:rus1207-01}Let
$\left(  X,\leq\right)  $ be a quasi-directed, partially ordered set, $d$ a
complete metric on $X$ and $T:X\rightarrow X$ a nondecreasing operator such
that%
\[
d\left(  T(x),T(y)\right)  \leq\varphi\left(  d(x,y)\right)  \quad\text{for
all }x,y\in X\text{ with }x\leq y\text{,}%
\]
for some $\varphi\in\Phi$. Assume that either

\begin{enumerate}
\item[(a1)] $T$ is continuous
\end{enumerate}

or

\begin{enumerate}
\item[(a2)] every nondecreasing and convergent sequence in $X$ is bounded from
above by its limit.
\end{enumerate}

If there exists $x_{0}\in X$ such that $x_{0}\leq T(x_{0})$, then $T$ has a
unique fixed point $x^{\ast}\in X$ and $T^{n}(x)\rightarrow x^{\ast}$ (as
$n\rightarrow\infty$) for all $x\in X$.
\end{theorem}

Based on Theorem \ref{th:rus1207-01}, we derive a similar result for the class
of mixed monotone operators, which will form the basis for the proof of one of
our main theorems (see Theorem \ref{th:rus1207-05}). For a related result due
to Lakshmikantham and \'{C}iri\'{c}, we point to \cite[Corollary
2.2]{Lakshmikantham2009}.

\begin{theorem}
\label{th:rus1207-02}Let $\left(  X,\leq\right)  $ be a bi-directed, partially
ordered set, $d$ a complete metric on $X$ and $A:X^{2}\rightarrow X$ a mixed
monotone operator such that%
\begin{equation}
d\left(  A(x,y),A(u,v)\right)  \leq\varphi\left(  \max\left\{
d(x,u),d(y,v)\right\}  \right)  \quad\text{for all }x,y,u,v\in X\text{ with
}x\leq u\text{, }y\geq v \label{eq:rus1207-04}%
\end{equation}
for some $\varphi\in\Phi$. Assume that either

\begin{enumerate}
\item[(b1)] $A$ is continuous
\end{enumerate}

or

\begin{enumerate}
\item[(b2)] every nondecreasing (respectively, nonincreasing) and convergent
sequence in $X$ is bounded from above (respectively, from below) by its limit.
\end{enumerate}

If there exist $x_{0},y_{0}\in X$ such that $x_{0}\leq A(x_{0},y_{0})$ and
$y_{0}\geq A(y_{0},x_{0})$, then there exists $x^{\ast}\in X$ such that
$(x^{\ast},x^{\ast})$ is the unique coupled fixed point of $A$ (hence,
$x^{\ast}$ is the unique fixed point of $A$) and $A^{n}(x,y)\rightarrow
x^{\ast}$ (as $n\rightarrow\infty$) for all $x,y\in X$.
\end{theorem}

\begin{proof}
Let $X^{2}$ be partially ordered by $(x,y)\preccurlyeq(u,v)\Leftrightarrow
x\leq u$, $y\geq v$. Since $\left(  X,\leq\right)  $ is bi-directed, it
follows that $(X^{2},\preccurlyeq)$ is quasi-directed (cf. \cite[Remark
2.6]{Rus2011}). Also, $\delta\left(  (x,y),(u,v)\right)  =\max\left\{
d(x,u),d(y,v)\right\}  $ ($x,y,u,v\in X$) defines a complete metric on $X^{2}$.

Now, let $T:X^{2}\rightarrow X^{2}$ be defined by $T(x,y)=\left(
A(x,y),A(y,x)\right)  $ ($x,y\in X$). It easily checks that $T$ is
nondecreasing, by the mixed-monotonicity of $A$. By switching $x$ with $v$,
and $y$ with $u$ in (\ref{eq:rus1207-04}), it follows that%
\begin{equation}
d\left(  A(y,x),A(v,u)\right)  \leq\varphi\left(  \max(d(x,u),d(y,v))\right)
\quad\text{for all }x,y,u,v\in X\text{ with }x\leq u\text{, }y\geq v.
\label{eq:rus1207-05}%
\end{equation}
Then together, (\ref{eq:rus1207-04}) and (\ref{eq:rus1207-05}) rewrite as%
\[
\delta\left(  T(x,y),T(u,v)\right)  \leq\varphi\left(  \delta\left(
(x,y),(u,v)\right)  \right)  \quad\text{for all }x,y,u,v\in X\text{ with
}(x,y)\preccurlyeq(u,v)\text{.}%
\]

Finally, (b1) implies (a1), and (b2) implies (a2) with $X^{2}$ in place of
$X$, by observing that a sequence $(x_{n},y_{n})$ in $X^{2}$ is nondecreasing
if and only if $(x_{n})$ is nondecreasing and $(y_{n})$ is nonincreasing in
$X$. Also, $(x_{0},y_{0})\preccurlyeq T(x_{0},y_{0})$.

Now, we can apply Theorem \ref{th:rus1207-01} for $(X^{2},\preccurlyeq)$,
$\delta$ and $T$; hence, $T$ has a unique fixed point $(x^{\ast},y^{\ast})$.
But then $(y^{\ast},x^{\ast})$ is also a fixed point of $T$; hence $y^{\ast
}=x^{\ast}$. Finally, it is straightforward to check that $T^{n}(x,y)=\left(
A^{n}(x,y),A^{n}(y,x)\right)  $ for all $x,y\in X$ and $n\geq0$, which
concludes the proof.
\end{proof}

\section{Main results}

In what follows, if not stated otherwise, $\left(  \mathcal{X},\mathcal{T}%
\right)  $ will be a partially monotone system of order $N\geq2$, with
$\mathcal{X}=\left\{  (X_{i},\leq_{i}):i\in\{1,2,\ldots,N\}\right\}  $ and
$\mathcal{T}=\left\{  T_{i}:X_{1}\times X_{2}\times\cdots\times X_{N}%
\rightarrow X_{i}:i\in\{1,2,\ldots,N\}\right\}  $). Also, let $X:=X_{1}\times
X_{2}\times\cdots\times X_{N}$ be endowed with the usual product quasi-order
$\leq$ defined in (\ref{eq:rus1207-02}) and let $T=(T_{1},T_{2},\ldots
,T_{N}):X\rightarrow X$. Also, for $x\in X$ and $j\in\{1,2,\ldots,N\}$, let
$x_{j}$ be the component in $X_{j}$ of $x$.

For each $i\in\{1,2,\ldots,N\}$, associate to $T_{i}$ the operator $\sigma
_{i}=(\sigma_{i,1},\sigma_{i,2},\ldots,\sigma_{i,N}):X^{2}\rightarrow X$ whose
components are defined as%
\begin{equation}
\sigma_{i,j}:X^{2}\rightarrow X_{j},\quad\sigma_{i,j}(x,y)=\left\{
\begin{array}
[c]{l}%
x_{j}\text{,\quad if }T_{i}\text{ is nondecreasing in the }j\text{-th
variable}\\
y_{j}\text{,\quad if }T_{i}\text{ is nonincreasing in the }j\text{-th
variable}%
\end{array}
\right.  (j\in\{1,2,\ldots,N\},~x,y\in X). \label{eq:rus1207-06}%
\end{equation}
Note that when $T_{i}$ is constant in the $j$-th variable (for some $j$), then
$\sigma_{i,j}$ can be chosen at convenience, but consistently, as if $T_{i}$
were either nondecreasing, or nonincreasing in the $j$-th variable (i.e.,
either $\sigma_{i,j}(x,y)=x_{j}$ for all $x,y\in X$, or $\sigma_{i,j}%
(x,y)=y_{j}$ for all $x,y\in X)$.

We also write $x\preccurlyeq_{i}y$ (for $x,y\in X$) when $\sigma_{i}%
(x,y)\leq\sigma_{i}(y,x)$, i.e.,
\begin{equation}
x\preccurlyeq_{i}y\Longleftrightarrow\left\{
\begin{array}
[c]{l}%
x_{j}\leq_{j}y_{j}\text{ if }T_{i}\text{ is nondecreasing in the }j\text{-th
variable}\\
x_{j}\geq_{j}y_{j}\text{ if }T_{i}\text{ is nonincreasing in the }j\text{-th
variable}%
\end{array}
,\right.  \quad\text{for all }j\in\{1,2,\ldots,N\}. \label{eq:rus1207-07}%
\end{equation}
This clearly defines an (alternative) quasi-order on $X$, subject to which
$T_{i}$ is nondecreasing in each of the variables.

With these notations, we prove the following result:

\begin{lemma}
\label{th:rus1207-03}Let $i\in\{1,2,\ldots,N\}$. The following properties take place:

\begin{enumerate}
\item[(c1)] $\sigma_{i}\left(  \sigma_{i}(x,y),\sigma_{i}(u,v)\right)
=\sigma_{i}(x,v)$ for all $x,y,u,v\in X$.

\item[(c2)] $\sigma_{i}^{2}=P_{X}$ , where the functional powers are
considered with respect to the $s$-composition.

\item[(c3)] For all $x,y,u,v\in X$: $\sigma_{i}(x,y)\preccurlyeq_{i}\sigma
_{i}(u,v)\Leftrightarrow\sigma_{i}(x,v)\leq\sigma_{i}(u,y)$.

\item[(c4)] For all $x,y\in X$: $\sigma_{i}(x,y)\preccurlyeq_{i}\sigma
_{i}(y,x)\Leftrightarrow x\leq y$.

\item[(c5)] $\sigma_{i}:(X^{2},\leq)\rightarrow(X,\leq)$ is nondecreasing,
i.e., $\sigma_{i}(x,y)\leq\sigma_{i}(u,v)$ whenever $x\leq u$ and $y\leq v$
($x,y,u,v\in X$).

\item[(c6)] $\sigma_{i}:(X^{2},\leq)\rightarrow(X,\preccurlyeq_{i})$ is mixed
monotone, i.e., $\sigma_{i}(x,y)\preccurlyeq_{i}\sigma_{i}(u,v)$ whenever
$x\leq u$ and $y\geq v$ ($x,y,u,v\in X$).

\item[(c7)] $T_{i}:(X,\preccurlyeq_{i})\rightarrow(X_{i},\leq)$ is
nondecreasing, i.e., $T_{i}(x)\leq T_{i}(y)$ whenever $x\preccurlyeq_{i}y$
($x,y\in X$).
\end{enumerate}
\end{lemma}

\begin{proof}
Fix $i\in\{1,2,\ldots,N\}$.

\begin{enumerate}
\item[(c1)] Let $x,y,u,v\in X$ arbitrary and $j\in\{1,2,\ldots,N\}$.

When $T_{i}$ is nondecreasing in the $j$-th variable, then $\sigma
_{i,j}\left(  \sigma_{i}(x,y),\sigma_{i}(u,v)\right)  =\sigma_{i,j}%
(x,y)=x_{j}=\sigma_{i,j}(x,v)$. Else, when $T_{i}$ is nonincreasing in the
$j$-th variable, then $\sigma_{i,j}\left(  \sigma_{i}(x,y),\sigma
_{i}(u,v)\right)  =\sigma_{i,j}(u,v)=v_{j}=\sigma_{i,j}(x,v)$, concluding the proof.

\item[(c2)] This follows from (c1), since $\sigma_{i}^{2}(x,y)=\sigma
_{i}\left(  \sigma_{i}(x,y),\sigma_{i}(y,x)\right)  =\sigma_{i}(x,x)=x=P_{X}%
(x,y)$ for all $x,y\in X$.

\item[(c3)] $\sigma_{i}(x,y)\preccurlyeq_{i}\sigma_{i}(u,v)$ is equivalent to
$\sigma_{i}\left(  \sigma_{i}(x,y),\sigma_{i}(u,v)\right)  \leq\sigma
_{i}\left(  \sigma_{i}(u,v),\sigma_{i}(x,y)\right)  $ by the definition of
$\preccurlyeq_{i}$, hence to $\sigma_{i}(x,v)\leq\sigma_{i}(u,y)$ by (c1), for
all $x,y,u,v\in X$.

\item[(c4)] This follows from (c3), by letting $u:=y$, $v:=x$.

\item[(c5)] This is obviously true, by the definition of $\sigma_{i}$.

\item[(c6)] If $x\leq u$ and $y\geq v$ ($x,y,u,v\in X$), then $\sigma
_{i}(x,y)\preccurlyeq_{i}\sigma_{i}(u,v)$ is equivalent, by (c3), with
$\sigma_{i}(x,v)\leq\sigma_{i}(u,y)$, which follows from (c5).

\item[(c7)] This is a clear consequence of the definition of $\preccurlyeq
_{i}$ (see (\ref{eq:rus1207-07})).
\end{enumerate}
\end{proof}

The following announced result establishes that $T$ is heterotone and shows
how to construct a mixed monotone operator $A$ such that (\ref{eq:rus1207-03})
is satisfied. This is an extension of one of our previous results \cite[Lemma
3.1]{Rus2008} established in the setting of ordered Banach spaces.

\begin{theorem}
\label{th:rus1207-04}The operator%
\begin{equation}
A=(A_{1},A_{2},\ldots,A_{N}):X^{2}\rightarrow X,\quad A_{i}=T_{i}\sigma
_{i}\quad(i\in\{1,2,\ldots,N\}) \label{eq:rus1207-08}%
\end{equation}
is mixed monotone (subject to $\leq$) and satisfies (\ref{eq:rus1207-03});
hence, $T$ is heterotone. Moreover,%
\[
\left(  A_{i}\ast\sigma_{i}\right)  (x,y)=T_{i}(x)\quad\text{for all }%
i\in\{1,2,\ldots,N\}\text{ and }x,y\in X\text{.}%
\]

\end{theorem}

\begin{proof}
Clearly, $A$ is correctly defined, since $T_{i}:X\rightarrow X_{i}$ and
$\sigma_{i}:X^{2}\rightarrow X$. The mixed monotonicity of $A$ follows by the
mixed monotonicity of every component $A_{i}=T_{i}\sigma_{i}:(X^{2}%
,\leq)\rightarrow(X_{i},\leq_{i})$, since $\sigma_{i}:(X^{2},\leq
)\rightarrow(X,\preccurlyeq_{i})$ is mixed monotone and $T_{i}:(X,\preccurlyeq
_{i})\rightarrow(X_{i},\leq_{i})$ is nondecreasing (by Lemma
\ref{th:rus1207-03}).

Property (\ref{eq:rus1207-03}) follows directly from (\ref{eq:rus1207-08}).
Moreover, for every $i\in\{1,2,\ldots,N\}$ and $x,y\in X$, it follows by (c2)
in Lemma \ref{th:rus1207-03} that%
\[
\left(  A_{i}\ast\sigma_{i}\right)  (x,y)=T_{i}\sigma_{i}(\sigma
_{i}(x,y),\sigma_{i}(y,x))=T_{i}\sigma_{i}^{2}(x,y)=T_{i}P_{X}(x,y)=T_{i}(x),
\]
which concludes the proof.
\end{proof}

\begin{example}
Consider (\ref{eq:rus1207-01}) for $N=3$:%
\[
\left\{
\begin{array}
[c]{c}%
x_{1}=T_{1}(x_{1},x_{2},x_{3})\\
x_{2}=T_{2}(x_{1},x_{2},x_{3})\\
x_{3}=T_{3}(x_{1},x_{2},x_{3})
\end{array}
\right.
\]
and assume that $T_{1}$ is nondecreasing in the first two argument and
nonincreasing in the last one, $T_{2}$ is nondecreasing in the last two
arguments and nonincreasing in the first one, while $T_{3}$ is nonincreasing
in all of the three arguments; we can represent this, for short, as%
\[
\left\{
\begin{array}
[c]{c}%
T_{1}:\nearrow\nearrow\searrow\\
T_{2}:\searrow\nearrow\nearrow\\
T_{3}:\searrow\searrow\searrow
\end{array}
.\right.
\]
In this case, the operators $\sigma_{i}$ are%
\[
\left\{
\begin{array}
[c]{c}%
\sigma_{1}(x,y)=(x_{1},x_{2},y_{3})\\
\sigma_{2}(x,y)=(y_{1},x_{2},x_{3})\\
\sigma_{3}(x,y)=(y_{1},y_{2},y_{3})
\end{array}
\right.  ,
\]
the partial orders $\preccurlyeq_{i}$ are defined as%
\[
\left\{
\begin{array}
[c]{c}%
x\preccurlyeq_{1}y\Leftrightarrow x_{1}\leq_{1}y_{1},~x_{2}\leq_{2}%
y_{2},~x_{3}\geq_{3}y_{3}\\
x\preccurlyeq_{2}y\Leftrightarrow x_{1}\geq_{1}y_{1},~x_{2}\leq_{2}%
y_{2},~x_{3}\leq_{3}y_{3}\\
x\preccurlyeq_{3}y\Leftrightarrow x_{1}\geq_{1}y_{1},~x_{2}\geq_{2}%
y_{2},~x_{3}\geq_{3}y_{3}%
\end{array}
,\right.
\]
while $A$ has the expression
\[
A(x,y)=%
\begin{pmatrix}
T_{1}(x_{1},x_{2},y_{3})\\
T_{2}(y_{1},x_{2},x_{3})\\
T_{3}(y_{1},y_{2},y_{3})
\end{pmatrix}
\]
for all $x=(x_{1},x_{2},x_{3})$ and $y=(y_{1},y_{2},y_{3})$ in $X$.
\end{example}

Though simple, Theorem \ref{th:rus1207-04} is a powerful tool that connects
every partially monotone system to a mixed monotone operator. In this way, we
can study (\ref{eq:rus1207-01}) by analyzing the fixed points of operator $A$
defined in (\ref{eq:rus1207-08}), e.g., using Theorem \ref{th:rus1207-02}. Our
following result enforces this assertion and shows how this can be used in practice.

\begin{theorem}
\label{th:rus1207-05}Let $(X_{i},\leq_{i})$ be partially ordered, bi-directed
and $d_{i}$ be a complete metric on $X_{i}$, for every $i\in\left\{
1,2,\ldots,N\right\}  $. Assume there exists $\varphi\in\Phi$ such that%
\begin{equation}
d_{i}(T_{i}(x),T_{i}(y))\leq\varphi\left(  \max\limits_{j}d_{j}(x_{j}%
,y_{j})\right)  \quad\text{for all }x,y\in X\text{ with }x\preccurlyeq
_{i}y\text{ and }i\in\{1,2,\ldots,N\}\text{.} \label{eq:rus1207-09}%
\end{equation}

Assume also that either

\begin{enumerate}
\item[(d1)] $T_{i}$ is continuous, for every $i\in\left\{  1,2,\ldots
,N\right\}  $
\end{enumerate}

or

\begin{enumerate}
\item[(d2)] every nondecreasing (respectively, nonincreasing) and convergent
sequence in $X_{i}$ is bounded from above (respectively, from below) by its
limit, for every $i\in\left\{  1,2,\ldots,N\right\}  $.
\end{enumerate}

If there exists $(x^{0},y^{0})\in X^{2}$ such that%
\begin{equation}
x_{i}^{0}\leq_{i}T_{i}\sigma_{i}(x^{0},y^{0}),\text{~}y_{i}^{0}\geq_{i}%
T_{i}\sigma_{i}(y^{0},x^{0})\quad\text{for all }i\in\{1,2,\ldots,N\},
\label{eq:rus1207-10}%
\end{equation}
then (\ref{eq:rus1207-01}) has a unique solution $x^{\ast}\in X$. Moreover,
for every $u^{0},v^{0}\in X$, the sequences $(u^{n})$, $(v^{n})$ in $X$
defined recursively by%
\[
u_{i}^{n+1}=T_{i}\sigma_{i}(u^{n},v^{n}),~v_{i}^{n+1}=T_{i}\sigma_{i}%
(v^{n},u^{n})\quad(i\in\{1,2,\ldots,N\},~n\geq0\mathbb{)}%
\]
are convergent to $x^{\ast}$ with respect to the product topology on $X$,
i.e., $\left(  u_{i}^{n}\right)  \rightarrow x_{i}^{\ast}$ and $\left(
v_{i}^{n}\right)  \rightarrow x_{i}^{\ast}$ as $n\rightarrow\infty$, for all
$i\in\left\{  1,2,\ldots,N\right\}  $.
\end{theorem}

\begin{proof}
It is enough to apply Theorem \ref{th:rus1207-02} for the mixed monotone
operator $A:X^{2}\rightarrow X$ defined in Theorem \ref{th:rus1207-04} to
obtain the conclusion.

Indeed, $(X,\leq)$ is bi-directed (as product of bi-directed partially ordered
sets; this follows as a simple extension of the argument in \cite[Remark
2.6]{Rus2011}). Also, (d1) implies (b1) (since the operators $\left\{
\sigma_{i}:i\in\{1,2,\ldots,N\}\right\}  $ are, clearly, continuous), while
(d2) implies (b2). Moreover, $x^{0}\leq A(x^{0},y^{0})$ and $y^{0}\geq
A(y^{0},x^{0})$ follow from (\ref{eq:rus1207-10}).

We only need to prove (\ref{eq:rus1207-04}) for the complete metric
$d(x,y)=\max\limits_{j}d_{j}(x_{j},y_{j})$ ($x,y\in X$), so let $x,y,u,v\in X$
with $x\leq u$, $y\geq v$ and fix $i\in\{1,2,\ldots,N\}$. With $s:=\sigma
_{i}(x,y)$ and $t:=\sigma_{i}(u,v)$, it follows that $s\preccurlyeq_{i}t$, by
(c6) in Lemma \ref{th:rus1207-03}; hence,%
\begin{equation}
d_{i}\left(  A_{i}(x,y),A_{i}(u,v)\right)  =d_{i}(T_{i}\sigma_{i}%
(x,y),T_{i}\sigma_{i}(u,v))=d_{i}(T_{i}(s),T_{i}(t))\leq\varphi\left(
\max\limits_{j}d_{j}(s_{j},t_{j})\right)  =\varphi\left(  d(s,t)\right)
\label{eq:rus1207-11}%
\end{equation}
by (\ref{eq:rus1207-09}). Since $d_{j}(s_{j},t_{j})\in\left\{  d_{j}%
(x_{j},u_{j}),d_{j}(y_{j},v_{j})\right\}  $ for all $j\in\{1,2,\ldots,N\}$, it
follows that $d(s,t)\leq\max\left\{  d(x,u),d(y,v)\right\}  $. Now
(\ref{eq:rus1207-04}) follows by (\ref{eq:rus1207-11}) and the fact that
$\varphi$ is nondecreasing. The proof is now complete.
\end{proof}

\begin{remark}
By letting $v^{0}:=u^{0}$ in the conclusion of Theorem \ref{th:rus1207-05}, it
follows that $u^{n}=v^{n}=T^{n}(u^{0})$ for all $n\geq0$; hence,
$T^{n}(x)\rightarrow x^{\ast}$ as $n\rightarrow\infty$ for all $x\in X$.
\end{remark}

Clearly, one can obtain other criteria for the existence and/or uniqueness of
solutions to (\ref{eq:rus1207-01}) by applying different (coupled) fixed point
results to the mixed monotone operator $A$. Further aspects will be delineated elsewhere.

\section{Applications}

\subsection{Application to a tripled fixed point problem}

Let $(X,\leq)$ be a partially ordered set and $F:X^{3}\rightarrow X$ such that
$F$ is nondecreasing in the first and the last variable, while nonincreasing
in the second one. In this context, Berinde and Borcut \cite{Berinde2011}
studied the existence and uniqueness of solutions to the tripled fixed point
problem:%
\begin{equation}
\left\{
\begin{array}
[c]{c}%
x=F(x,y,z)\\
y=F(y,x,y)\\
z=F(z,y,x)
\end{array}
\right.  \quad(x,y,z\in X). \label{eq:rus1207-12}%
\end{equation}
Turinici \cite{Turinici2011} showed that the arguments in \cite{Berinde2011}
can be greatly simplified by looking at (\ref{eq:rus1207-12}) as the fixed
point problem for the nondecreasing operator%
\begin{equation}
T:X^{3}\rightarrow X^{3},\quad T(x,y,z)=\left(
\begin{array}
[c]{c}%
F(x,y,z)\\
F(y,x,y)\\
F(z,y,x)
\end{array}
\right)  \quad(x,y,z\in X) \label{eq:rus1207-13}%
\end{equation}
subject to the partial order $(x,y,z)\preccurlyeq(u,v,w)\Leftrightarrow x\leq
u,y\geq v,z\leq w$.

It is not hard to see that neither of the arguments in
\cite{Berinde2011,Turinici2011} work for the following modified tripled fixed
point problem%
\begin{equation}
\left\{
\begin{array}
[c]{c}%
x=F(x,y,z)\\
y=F(y,x,z)\\
z=F(z,y,x)
\end{array}
\right.  \quad(x,y,z\in X), \label{eq:rus1207-14}%
\end{equation}
simply because the associated operator $T$ (defined in a similar way as in
(\ref{eq:rus1207-13})) is not nondecreasing with respect to any of the partial
orders on $X^{3}$ obtained as products of $\leq$ and/or $\geq$.

Clearly, (\ref{eq:rus1207-14}) is the fixed point problem for a partially
monotone system of order $3$, where $T_{1},T_{2},T_{3}:X^{3}\rightarrow X$
given by%
\begin{equation}
\left\{
\begin{array}
[c]{c}%
T_{1}(x,y,z)=F(x,y,z)\\
T_{2}(x,y,z)=F(y,x,z)\\
T_{3}(x,y,z)=F(z,y,x)
\end{array}
\right.  \quad(x,y,z\in X) \label{eq:rus1207-15}%
\end{equation}
are coordinate-wise monotone: $T_{1}$ and $T_{3}$ are both nondecreasing in
$x$ and $z$, while nonincreasing in $y$; $T_{2}$ is nonincreasing in $x$,
while nondecreasing in $y$ and $z$.

If $d$ is a metric on $X$, then a simple analysis on the contractive condition
(\ref{eq:rus1207-09}) shows that if $T_{1}$ satisfies it, then the same
happens for $T_{2}$ and $T_{3}$. In this context, by rewriting
(\ref{eq:rus1207-10}) in terms of $F$, the following result follows as a
direct consequence of Theorem \ref{th:rus1207-05}:

\begin{corollary}
\label{th:rus1207-06}Let $(X,\leq)$ be a bi-directed partially ordered set,
$d$ a complete metric on $X$ and $F:X^{3}\rightarrow X$ such that $F$ is
nondecreasing in the first and the last variable, while nonincreasing in the
second one. Assume there exists $\varphi\in\Phi$ such that%
\begin{equation}
d(F(x,y,z),F(u,v,w))\leq\varphi\left(  \max\left\{
d(x,u),d(y,v),d(z,w)\right\}  \right)  \quad\text{for all }x,y,z,u,v,w\in
X\text{ with }x\leq u,y\geq v,z\leq w. \label{eq:rus1207-16}%
\end{equation}

Assume also that either

\begin{enumerate}
\item[(e1)] $F$ is continuous
\end{enumerate}

or

\begin{enumerate}
\item[(e2)] every nondecreasing (respectively, nonincreasing) and convergent
sequence in $X$ is bounded from above (respectively, from below) by its limit.
\end{enumerate}

If there exist $x_{0},y_{0},z_{0},u_{0},v_{0},w_{0}\in X$ such that%
\begin{equation}
\left\{
\begin{array}
[c]{l}%
x_{0}\leq F(x_{0},v_{0},z_{0});\\
y_{0}\geq F(y_{0},u_{0},z_{0});\\
z_{0}\leq F(z_{0},v_{0},x_{0});
\end{array}%
\begin{array}
[c]{l}%
u_{0}\geq F(u_{0},y_{0},w_{0})\\
v_{0}\leq F(v_{0},x_{0},w_{0})\\
w_{0}\geq F(w_{0},y_{0},u_{0})
\end{array}
\right.  \text{,} \label{eq:rus1207-17}%
\end{equation}
then (\ref{eq:rus1207-15}) has a unique solution.
\end{corollary}

\subsection{Application to a periodic boundary value system}

In what follows, we study the existence and uniqueness of the solution to a
periodic boundary value system (PBVS), as an application to Corollary
\ref{th:rus1207-06}.

Consider the following first-order PBVS:
\begin{equation}
\left\{
\begin{array}
[c]{l}%
x^{\prime}(t)=f(t,x(t),y(t),z(t))\\
y^{\prime}(t)=f(t,y(t),x(t),z(t))\\
z^{\prime}(t)=f(t,z(t),y(t),x(t))
\end{array}
\right.  \quad\text{for all }t\in I=[0,T]\text{,\quad with }\left\{
\begin{array}
[c]{c}%
x(0)=x(T)\\
y(0)=y(T)\\
z(0)=z(T)
\end{array}
\right.  . \label{eq:rus1207-18}%
\end{equation}
where $T>0$ and $f:I\times\mathbb{R}^{3}\rightarrow\mathbb{R}$, under the assumptions:

\begin{enumerate}
\item[(f1)] $f$ is continuous;

\item[(f2)] there exists $\lambda>0$ and $\varphi\in\Phi$ such that%
\[
-\lambda(u-x)\leq f(u,v,w)-f(x,y,z)\leq\lambda\left(  \varphi\left(
\max\left\{  u-x,y-v,w-z\right\}  \right)  -(u-x)\right)  \quad
\]
for all $x,y,z,u,v,w\in\mathbb{R}$ with $x\leq u$, $y\geq v$, $z\leq w$.
\end{enumerate}

We say that $%
\begin{pmatrix}
x & y & z\\
u & v & w
\end{pmatrix}
$ is a coupled lower and upper solution to (\ref{eq:rus1207-18}) if
$x,y,z,u,v,w\in C^{1}(I)$ and%
\[
\left\{
\begin{array}
[c]{ll}%
x^{\prime}(t)\leq f(t,x(t),v(t),z(t)); & u^{\prime}(t)\geq
f(t,u(t),y(t),w(t))\\
y^{\prime}(t)\geq f(t,y(t),u(t),z(t)); & v^{\prime}(t)\leq
f(t,v(t),x(t),w(t))\\
z^{\prime}(t)\leq f(t,z(t),v(t),x(t)); & w^{\prime}(t)\geq f(t.w(t),y(t),u(t))
\end{array}
\right.  \text{for all }t\in I=[0,T]\text{, and }\left\{
\begin{array}
[c]{cl}%
x(0)\leq x(T); & u(0)\geq u(T)\\
y(0)\geq y(T); & v(0)\leq v(T)\\
z(0)\leq z(T); & w(0)\geq w(T)
\end{array}
\right.  \text{.}%
\]
Clearly, if $(x,y,z)$ is a solution to (\ref{eq:rus1207-18}), then $%
\begin{pmatrix}
x & y & z\\
x & y & z
\end{pmatrix}
$ is a coupled lower and upper solution to (\ref{eq:rus1207-18}).

Recall that if $x\in C^{1}(I)$ satisfies%
\begin{equation}
\left\{
\begin{array}
[c]{l}%
x^{\prime}(t)=h(t),\quad t\in I\\
x(0)=x(T)
\end{array}
\right.  , \label{eq:rus1207-19}%
\end{equation}
with $h\in C(I)$, then for every $\lambda\neq0$,%
\begin{equation}
x(t)=\int_{0}^{T}G_{\lambda}(t,s)(h(s)+\lambda x(s))~\mathrm{d}s\quad\text{for
all }t\in I\text{,} \label{eq:rus1207-20}%
\end{equation}
where%
\[
G_{\lambda}(t,s)=\left\{
\begin{tabular}
[c]{ll}%
$\dfrac{\mathrm{e}^{\lambda(T+s-t)}}{\mathrm{e}^{\lambda T}-1},$ & if $0\leq
s<t\leq T$\\
& \\
$\dfrac{\mathrm{e}^{\lambda(s-t)}}{\mathrm{e}^{\lambda T}-1},$ & if $0\leq
t<s\leq T$%
\end{tabular}
\right.  \text{.}\
\]
Conversely, if $x\in C(I)$ satisfies (\ref{eq:rus1207-20}) for some
$\lambda\neq0$, then $x\in C^{1}(I)$ and verifies (\ref{eq:rus1207-19}).

Also, if $x\in C^{1}(I)$ is a lower solution for (\ref{eq:rus1207-19}), i.e.,%
\[
\left\{
\begin{array}
[c]{l}%
x^{\prime}(t)\leq h(t),\quad t\in I\\
x(0)\leq x(T)
\end{array}
\right.  ,
\]
then $x$ is a lower solution for (\ref{eq:rus1207-20}), i.e.,%
\[
x(t)\leq\int_{0}^{T}G_{\lambda}(t,s)(h(s)+\lambda x(s))~\mathrm{d}%
s\quad\text{for all }t\in I\text{.}%
\]
Similarly, if $x$ is an upper solution for (\ref{eq:rus1207-19}), i.e.,%
\[
\left\{
\begin{array}
[c]{l}%
x^{\prime}(t)\geq h(t),\quad t\in I\\
x(0)\geq x(T)
\end{array}
\right.  ,
\]
then $x$ is an upper solution for (\ref{eq:rus1207-20}), i.e.,%
\[
x(t)\geq\int_{0}^{T}G_{\lambda}(t,s)(h(s)+\lambda x(s))~\mathrm{d}%
s\quad\text{for all }t\in I\text{.}%
\]

From these results, it follows that the PBVS (\ref{eq:rus1207-18}) is
equivalent to the tripled fixed point problem (\ref{eq:rus1207-14}), with
$X=C(I)$ and
\begin{equation}
F:X^{3}\rightarrow X,\quad F(x,y,z)(t)=\int_{0}^{T}G_{\lambda}(t,s)\left(
f(x(s),y(s),z(s))+\lambda x(s)\right)  ~\mathrm{d}s\quad\text{for all }t\in
I\text{ and }x,y,z\in C(I)\text{;} \label{eq:rus1207-21}%
\end{equation}
additionally, if $%
\begin{pmatrix}
x_{0} & y_{0} & z_{0}\\
u_{0} & v_{0} & w_{0}%
\end{pmatrix}
$ is a coupled lower and upper solution to (\ref{eq:rus1207-18}), then
(\ref{eq:rus1207-17}) follows.

Now, we are able to formulate and prove the main result in this section.

\begin{theorem}
Under the assumptions (f1) and (f2), if (\ref{eq:rus1207-18}) has a coupled
lower and upper solution, then (\ref{eq:rus1207-18}) has a unique solution.
\end{theorem}

\begin{proof}
Due to the previous remarks, it is sufficient to verify the conditions in
Corollary \ref{th:rus1207-06} for $X:=C(I)$, partially ordered by $x\preceq
y\Leftrightarrow x(t)\leq y(t)$ for all $t\in I~($for $x,y\in X)$, where the
complete metric $d$ is induced by the sup-norm on $X$, i.e., $d(x,y)=\sup
_{t\in I}\left\vert x(t)-y(t)\right\vert $ $(x,y\in X)$ and the operator $F$
is defined by (\ref{eq:rus1207-21}).

It is know that $(X,\preceq)$ is a lattice, hence $(X,\preceq)$ is
bi-directed. We check now that $F$ is nondecreasing in the first and the last
variable, while nonincreasing in the second one and that it verifies the
contraction condition (\ref{eq:rus1207-16}). Indeed, for all $x,y,z,u,v,w\in
X$ with $x\leq u,y\geq v,z\leq w$, it follows from the first inequality in
(f2) and the positivity of $G_{\lambda}$ (for $\lambda>0$) that
\[
F(u,v,w)(t)-F(x,y,z)(t)=\int_{0}^{T}G_{\lambda}(t,s)\left(
f(u(s),v(s),w(s))+\lambda u(s)-f(x(s),y(s),z(s))-\lambda x(s)\right)
~\mathrm{d}s\geq0\text{\quad for all }t\in I\text{.}%
\]
Now, from the last inequality in (f2), we have%
\begin{align*}
F(u,v,w)(t)-F(x,y,z)(t)  &  \leq\lambda\int_{0}^{T}G_{\lambda}(t,s)\cdot
\varphi\left(  \max\left\{  u(s)-x(s),y(s)-v(s),w(s)-z(s)\right\}  \right)
~\mathrm{d}s\\
&  \leq\lambda\left(  \int_{0}^{T}G_{\lambda}(t,s)~\mathrm{d}s\right)
\cdot\varphi\left(  \max\left\{  d(x,u),d(y,v),d(z,w)\right\}  \right) \\
&  \leq\varphi\left(  \max\left\{  d(x,u),d(y,v),d(z,w)\right\}  \right)
\text{\quad for all }t\in I\text{,}%
\end{align*}
where we used that $\varphi$ is nondecreasing and $\int_{0}^{T}G_{\lambda
}(t,s)~\mathrm{d}s=\frac{1}{\lambda}$ for all $t\in I$; this concludes
(\ref{eq:rus1207-16}).

Finally, both (e1) and (e2) are satisfied, while (\ref{eq:rus1207-17}) follows
by the assumption that (\ref{eq:rus1207-18}) has a coupled lower and upper
solution. The proof is now complete.
\end{proof}

\section{Some conclusions and final remarks}

\begin{remark}
\label{th:rus1207-07}In \cite{Turinici2011}, Turinici studied
(\ref{eq:rus1207-01}) under the assumption that $\{T_{i}:i\in\{1,2,\ldots
,N\}\}$ are all coordinate-wise nondecreasing, by analyzing the fixed point
problem for the nondecreasing operator $T=(T_{1},T_{2},\ldots,T_{N}%
):X\rightarrow X$, subject to the usual product quasi-order on $X$ (defined in
(\ref{eq:rus1207-02})).

Clearly, the methods and results in \cite{Turinici2011} also apply when any of
the quasi-orders $\leq_{i}$ are reversed (i.e., $\leq_{i}$ is replaced with
the reversed quasi-order $\geq_{i}$), but as one can easily see, it is not
possible (by this method only) to cover all the existing combinations of
coordinate-wise monotone operators; it is enough to consider $N=2$ to notice
there are only two types (of the total 16) of partially monotone systems for
which the associated operator $T$ can be made nondecreasing, by eventually
reversing $\leq_{1}$ and/or $\leq_{2}$. On the same note, a simple
combinatorial analysis shows that for any given $N$, there are $2^{N^{2}}$
types of partially monotone systems of order $N$ (based on a given family of
quasi-ordered sets), yet for only $2^{N-1}$ of them the associated operator
$T$ can be made nondecreasing by a (selective) reversing of the quasi-orders.

Based on this fact, we can conclude that our method for studying
(\ref{eq:rus1207-01}) is justified and can not be superseded by the simpler
method in \cite{Turinici2011}.
\end{remark}

\begin{remark}
Due to the result established in Theorem \ref{eq:rus1207-04}, one can obtain
any number of criteria for the existence and/or uniqueness of solutions to
(\ref{eq:rus1207-01}) by applying different (coupled) fixed point results to
the associated mixed monotone operator. In this context, Theorem
\ref{th:rus1207-05} should be seen as one of the many possible answers to
(\ref{eq:rus1207-01}), hence leaving space for many developments and improvements.
\end{remark}

\begin{remark}
In the recent years, there has been an increased interest in generalizing the
notions of fixed point and coupled fixed point for operators with monotonic
properties to other concepts like tripled fixed point \cite{Berinde2011},
quadrupled fixed point \cite{Karapinar2012}, fixed point of $N$-order
\cite{Berzig2012} or multidimensional fixed point \cite{Roldan2012} (where the
letter contains all the others). All these new concepts (and any of their
possible variations) can be seen as solutions to some particular partially
monotone system, as previously exemplified; hence, can be studied in a single, unified and more
clear approach using the methods described in this paper.
\end{remark}

\section*{Acknowledgement}

The author is grateful for the financial support provided by the Sectoral
Operational Programme Human Resources Development 2007-2013 of the Romanian
Ministry of Labor, Family and Social Protection through the Financial
Agreement POSDRU/89/1.5/S/62557.

\bibliographystyle{elsart-num-sort}
\bibliography{RUS_20120830}

\end{document}